\documentclass[12pt,a4paper]{amsart}
\usepackage{amssymb}
\usepackage{amsmath}
\usepackage{amscd}
\usepackage{latexsym}

\headsep=1cm \numberwithin{equation}{section}
\newtheorem{theorem}{Theorem}[section]
\newtheorem{lemma}[theorem]{Lemma}
\newtheorem{proposition}[theorem]{Proposition}
\newtheorem{corollary}[theorem]{Corollary}

\theoremstyle{definition}
\newtheorem{definition}[theorem]{Definition}
\theoremstyle{remark}
\newtheorem{remark}[theorem]{Remark}

\newtheorem{Definition and Notation}[theorem]{Definition and Notation}
\theoremstyle{Definition and Notation}

\newcommand{\coker}{\operatorname{coker}}
\newcommand{\Min}{\operatorname{Min}}
\newcommand{\Max}{\operatorname{Max}}

\newcommand{\Spec}{\operatorname{Spec}}

\newcommand{\HH}{\operatorname{H}}

\newcommand{\depth}{\operatorname{depth}}

\newcommand{\height}{\operatorname{ht}}

\newcommand{\id}{\operatorname{id}}

\newcommand{\Ext}{\mathrm{Ext}}
\newcommand{\Tor}{\operatorname{Tor}}
\newcommand{\Extindex}{\mathrm{Ext}\text{-}\mathrm{index}}
\newcommand{\Extgap}{\mathrm{Ext}\text{-}\mathrm{gap}}

\newcommand{\D}{\operatorname{D}}

\newcommand{\BN}{\Bbb N}

\newcommand{\fm}{\frak{m}}
\newcommand{\fp}{\frak{p}}

\newcommand{\fn}{\frak{n}}

\begin{document}

\title[On Ext-indices of ring extensions]
{On Ext-indices of ring extensions}

\author[S. Nasseh]{Saeed Nasseh}
\address{S. Nasseh, Department of Mathematics, Shahid Beheshti University,
Tehran, Iran   -and-   Institute for Studies in Theoretical Physics and
Mathematics, 19395-5746, Tehran, Iran.}
\email{saeed$\_$naseh@mail.ipm.ir}

\author[Y. Yoshino]{Yuji Yoshino}
\address{Y. Yoshino, Department of Mathematics, Faculty of Science, Okayama University, 700-8530, Okayama, Japan.}
\email{yoshino@math.okayama-u.ac.jp}


\begin{abstract}
In this paper we are concerned with the finiteness property of Ext-indices of several ring extensions. 
In this direction, we introduce some conjectures and discuss the relationship of them. 
Also we give affirmative answers to these conjectures in some special cases. 
Furthermore, we prove that the trivial extension of an Artinian local ring by its residue class field is always of finite Ext-index and we show that the Auslander-Reiten conjecture is true for this type of rings.
\end{abstract}


\maketitle


\section{introduction}
Throughout the paper,  all rings are assumed to be commutative Noetherian rings with unity. 

Let  $R$  be a ring. 
According to \cite{AY}, given nonzero $R$-modules $M$ and $N$, we define $p^R(M,N)$ by the following equality: 
$$
p^R (M,N)= \sup \{ i \in \BN  \ | \  \Ext_R^i(M,N) \neq  0 \} \ (\leq \infty).
$$
And we define the Ext-index of the ring  $R$,  denoted by $\Extindex (R)$, to be  the supremum of finite values  of  $p^R (M, N)$ for finitely generated  $R$-modules  $M$ and $N$, i.e.
$$
\begin{aligned}
\Extindex (R) = \sup \{ p^R(M,N) \ | \ &\text{$M$ and $N$ are finitely generated} \\
&\text{$R$-modules with } \ p^R(M,N) < \infty \  \}. 
\end{aligned}
$$

\begin{definition}
We say that the ring  $R$  is {\bf of finite Ext-index} if it satisfies 
$\Extindex (R) < \infty$. 
Following the paper \cite{HJ}, we often call  $R$  an AB ring if it is a Gorenstein local ring of finite Ext-index.  
\end{definition}

Recall that the following rings are examples known to be of finite Ext-index: 

\vspace{4pt} 
\noindent
Complete intersections (\cite[Corollary 3.5]{HJ}), 
Golod rings (\cite[Proposition 1.4]{JS}), 
Gorenstein local rings with minimal multiplicity (\cite[Theorem 3.6]{HJ}), 
and Gorenstein local rings with codimension at most 4 (\cite[Theorem 3.4]{S}). 
\vspace{4pt}


Note also from  \cite{JS} that there exists an example of an Artinian Gorenstein  local ring which is not AB.

\vspace{6pt}

In this paper we are mainly concerned with the finiteness property of Ext-indices of several ring extensions. 
This is motivated by the following conjectures, all of which seem to be open. 
(See also \cite{CH}.)

\par\vspace{6pt}\noindent {\bf Conjecture (L)} : 
Let  $R$ be a ring and let $\fp \in \Spec (R)$. 
If  $R$  is of finite Ext-index, then so would be the \underline{l}ocalization  $R_{\fp}$.  

\par\vspace{6pt}\noindent {\bf Conjecture (E)} : 
Let  $R$  be an algebra over a field $k$ and let  $\ell$  be a finitely generated \underline{e}xtension field  of  $k$. 
If  $R$  is of finite Ext-index, then so would be the ring $R \otimes _k \ell$.

\par\vspace{6pt}\noindent {\bf Conjecture (P)} : 
Let  $R$  be a ring. 
If  $R$  is of finite Ext-index, then so would be the \underline{p}olynomial ring $R[x]$.
\vspace{6pt}


In Section 2, after making some preliminaries, we discuss the relationship among these conjectures (Proposition \ref{relation of conjectures}). 
We shall also give some of the obvious cases for the above conjectures.

In Section 3, we are interested in the trivial extension $R(k)$  of an Artinian local ring  $R$ with its residue class field $k$. 
Surprisingly enough, we prove that  $R(k)$ is always of finite Ext-index (Corollary \ref{first cor}). 
Furthermore we can show that the Auslander-Reiten conjecture is true for the rings of this type (Corollary \ref{AR}).

In Section 4, we are interested in how the finiteness of Ext-index is preserved  by a base field extension for algebras. 
To be precise let  $R$ be a finite dimensional algebra over a field $k$ and we consider a transcendental extension $k(x)$  of  $k$. 
We show under a mild assumption that if  $R$  is of finite Ext-index then so is  the extended ring  $R \otimes _k k(x)$. 
See Theorem \ref{k(x)} for the detail. 
We also give some variants of this theorem in Theorems \ref{Extgap} and \ref{zeta}. 

For unexplained notation and terminologies in the paper, see the books \cite{BH}, \cite{E}, \cite{M} and \cite{W}.

\section{Preliminaries}

We recall some of the basic facts concerning the Ext-indices. 

\begin{lemma}\label{first lemma}  
\begin{itemize}
\item[] 
\item[$(1)$]
Let  $R \to S$  be a faithfully flat ring homomorphism. 
Then the inequality  $\Extindex (R) \leq \Extindex (S)$ holds. 
In particular, if  $S$  is of finite Ext-index, then so is  $R$. 

\item[$(2)$]
Let  $R = R_1 \times R_2$  be a product of rings. 
Then we have an equality 
$$
\Extindex (R) = \sup \{ \Extindex (R_1), \  \Extindex (R_2) \}.  
$$
In particular,  $R$  is of finite Ext-index if and only if so are the both of   $R_1$ and  $R_2$.

\item[$(3)$] 
Let  $x$  be a non-zero divisor of  $R$. 
Then the inequality $\Extindex (R/xR) \leq \Extindex (R) -1$ holds. 
In particular, if  $R$  is of finite Ext-index, then so is  $R/xR$. 

\item[$(4)$] 
Let $(R, \fm)$  be a Cohen-Macaulay local ring with dualizing module. 
And let  $x$  be a non-zero divisor of $R$ that belongs to $\fm$. 
If  $R/xR$  is of finite Ext-index, then so is  $R$. 
\end{itemize}
\end{lemma}

\begin{proof}
(1) and (2) are proved straightforward only from the definition. 
See \cite[Proposition 3.5]{CH} for (1) and \cite[Proposition 3.3]{CH} for (2). 
For  (3) and (4),  refer to \cite[Proposition 3.3(1)]{HJ}, \cite[Lemma 3.4]{CH} and \cite[Propositions 4.2]{CH}.
\end{proof}

In the following  we give an obvious case of Conjecture (L). 
In the lemma, $\Max (R)$  (resp. $\Min (R)$)  denotes the set of all maximal (resp. minimal prime) ideals of  $R$.

\begin{lemma}\label{maxmin} 
Let  $\fm \in \Max(R) \cap \Min (R)$. 
Then we have
$$
\Extindex (R_{\fm}) \leq \Extindex (R).  
$$
In particular, if $R$ is of finite Ext-index, then so is $R_{\fm}$ for each $\fm \in \Max (R) \cap \Min(R)$.
\end{lemma}

\begin{proof}
Let  $(0) = Q_1  \cap Q_2$  be  an irredundant primary decomposition, where  $Q_1$  is an $\fm$-primary component and  $Q_2$  is the intersection of the components  belonging to other primes.  
Since  $\fm$  is a maximal ideal, we have  $Q_1 + Q_2 = R$, hence  $R \cong R/Q_1 \times R/Q_2$. 
Noting that  $(0)_{\fm} = {Q_1}_{\fm}$, we see that 
$R_{\fm} \cong (R/Q_1)_{\fm} \cong  R/Q_1$. 
Therefore the lemma follows from Lemma \ref{first lemma}(2).
\end{proof}

\begin{corollary}\label{cor maxmin} 
If $R$ is an Artinian ring of finite Ext-index, then so is $R_{\fp}$ for every prime ideal $\fp$ of $R$. 
\end{corollary}

\begin{lemma}\label{localization} 
Let $R$ be a Cohen-Macaulay ring of finite Ext-index. 
Then $R_{\fm}$ is of finite Ext-index for every maximal ideal $\fm$ of $R$.
\end{lemma}

\begin{proof}
Let $d = \height (\fm) \geq 0$. 
We can take a regular sequence  $x_1 , \ldots , x_d  \in \fm$  on $R$. 
Then  $R/(x_1,...,x_d)R$ is of finite Ext-index by \ref{first lemma}(3). 
And thus  $(R/(x_1,...,x_d)R)_{\fm}$  is an Artinian ring of finite Ext-index by Lemma  \ref{maxmin}. 
Apply  \ref{first lemma}(4)  to this ring that is isomorphic to $\widehat{(R_{\fm})}/(x_1, \ldots , x_d)\widehat{(R_{\fm})}$, and we see that $\widehat{(R_{\fm})}$  is of finite Ext-index. 
Finally, by virtue of  \ref{first lemma}(1), we have the finiteness of Ext-index of  $R_{\fm}$.  
\end{proof}

The following is a corollary of the proof above.

\begin{corollary}\label{completion}
Let  $R$  be a Cohen-Macaulay local ring of finite Ext-index.
Then the completion  $\widehat R$  is also of finite Ext-index.
\end{corollary}

\begin{lemma}\label{Extindex=dim}
Let  $R$ be a Gorenstein ring of finite Ext-index and suppose that  $\dim (R) < \infty$. 
Then the equality  $\Extindex (R) =  \dim (R)$ holds.
\end{lemma}

Before proving this, we should remark that if  $R$ is an AB ring then the equality was shown in \cite[Proposition 3.2]{HJ}. 
Also this has been proved by Mori \cite[Corollary 3.3]{Mo} including non-commutative cases. 
We give below the proof for the convenience of the reader.

\begin{proof}
Let  $\fm$ be a maximal ideal of $R$. 
By \ref{localization},  we know that $R_{\fm}$ is an AB ring. 
It follows from the above mentioned result of \cite{HJ} that  $\Extindex (R_{\fm})= \height (\fm) \leq \dim (R)$. 

Now let $M$ and $N$ be finitely generated $R$-modules with $\Ext^i_R(M,N)=0$  for  $i\gg 0$.
Then we have obviously $\Ext^i_{R_{\fm}}(M_{\fm},N_{\fm}) =0$ for $i\gg 0$,  
hence the equality holds for all $i> \dim (R) \geq \Extindex (R_{\fm})$. 
This is true for any maximal ideal $\fm$. 
Hence we have  $\Ext^i_R(M,N)=0$  for $i > \dim (R)$. 
Therefore, it is concluded that  $\Extindex (R) \leq \dim (R)$. 

On the other hand, since  $\dim (R)= \id (R)$, we can find a finitely generated $R$-module $L$ such that $\Ext _R^d(L,R) \neq 0$, and $\Ext_R^i(L,R)=0$ for all $i> d = \dim (R)$. 
This implies that $\dim (R) \leq \Extindex (R)$. 
\end{proof}

\begin{lemma}\label{finite Krull dim} 
Let  $R$ be a Gorenstein ring of finite Krull dimension and suppose  $R_{\fm}$ is of finite Ext-index for each $\fm \in \Max (R)$. 
Then $R$ is of finite Ext-index.
\end{lemma}

\begin{proof}
In fact,  completely as in the same way as in the proof of \ref{Extindex=dim} we can prove the following inequalities: 
$$
\Extindex (R) \leq \sup \{\Extindex (R_{\fm}) \ | \ \fm \in \Max (R) \} \leq \dim (R) < \infty.
$$
\end{proof}

We observe the relationship among Conjectures (L), (E) and (P) introduced in Section 1.

\begin{proposition}\label{relation of conjectures}  
\begin{itemize}
\item[]
\item[$(1)$] 
Suppose that Conjecture (P)  is true for all Cohen-Macaulay local rings  $R$  of dimension one. 
Then Conjecture (L) is true for all Cohen-Macaulay local rings $R$  of any dimension with dualizing module.

\item[$(2)$] 
Suppose that Conjecture (P)  is true for a $k$-algebra $R$. 
Then Conjecture  (E)  is true for $R$ and for all simple algebraic extensions  $\ell$ of $k$.

\item[$(3)$]
Suppose that Conjectures (L) and (E) are true for all Gorenstein rings containing field.  
Then Conjecture (P) is true for all Gorenstein rings of finite Krull dimension that contain fields. 
\end{itemize}
\end{proposition}

\begin{proof}
(1) 
Suppose  (P)  holds for all Cohen-Macaulay local rings of dimension one. 
Let  $(R, \fm)$  be a Cohen-Macaulay local ring with dualizing module 
 and let  $\fp \in \Spec (R)$, and assume that  $R$  is of finite Ext-index. 
To prove that $R_{\fp}$  is of finite Ext-index, by induction on $\height (\fm/\fp)$, we may assume that  $\height  (\fm/\fp) =1$. 
Take a maximal regular sequence $\{ x_1, \ldots , x_h \}$  in  $\fp$  (so that  $h = \height (\fp)$), and consider the residue ring  $\overline R = R/(x_1, \ldots , x_h)$. 
By virtue of Lemma \ref{first lemma}(3) and (4), replacing  $R$  by $\overline R$, we may assume that  $R$  is of one dimension and  $\fp \in \Min (R)$. 
Then take a non-zero divisor  $a \in \fm$. 
Since  $R[x]$  is of finite Ext-index, it follows from Lemma \ref{first lemma}(3) that  $R_{a} \cong  R[x]/(ax-1)$  is also of finite Ext-index. 
Since $R_{a}$  is Artinian and  $a \not\in \fp$,  the localization  $R_{\fp}$ is of finite Ext-index as well, by Corollary \ref{cor maxmin}.  

(2) 
Let  $R$  be a $k$-algebra of finite Ext-index. 
Assume $\ell = k ( \alpha )$  is a simple algebraic extension of  $k$. 
Let  $f(x)$  be the minimal polynomial of  $\alpha$  over  $k$. 
Then we have $R \otimes _k \ell \cong  R[x]/(f(x))$. 
Since we are assuming that  $R[x]$ is of finite Ext-index,  $R[x]/(f(x))$ is of finite Ext-index as well, by Lemma \ref{first lemma}(3). 

(3)
Let  $R$  be a Gorenstein ring of finite Ext-index that contains a field. 
To prove that the polynomial ring  $R[x]$  is of finite Ext-index, 
 we only have to show that  $R[x]_{\frak M}$ is so for each maximal ideal $\frak M$  of  $R[x]$. 
(See Lemma \ref{finite Krull dim}.)
Set  $\fp = \frak M \cap R$ and we have that  $R_{\fp}$ is of finite Ext-index by the assumption that (L) is true for $R$. 
Replacing  $R$  with  $R_{\fp}$, we may assume that  $(R, \fm)$  is a Gorenstein local ring of finite Ext-index and that  $\frak M \cap R = \fm$. 
By virtue of Lemma \ref{first lemma}(1) and Corollary \ref{completion}, we may also assume that  $R$  is a complete local ring. 
Since  $R$ contains a field, $R$  has  a coefficient field $k$. 
Then it is obvious that there is an irreducible polynomial $f(x) (\not= 0) \in k[x]$  with  $\frak M = (\fm , f(x))R[x]$. 
By Lemma \ref{first lemma}(4), the finiteness of Ext-index for  $R[x]_{\frak M}$  follows from that for  $\left( R[x]/(f(x)) \right)_{\frak M}$. 
But the last ring is a localization of $R \otimes _k k[x]/(f(x))$, which is of finite Ext-index by the validity of  (E) and (L). 
\end{proof}

As to Conjecture (P) we give an affirmative answer in a special case.

\begin{proposition}\label{poly} 
Let  $R$  be an Artinian Gorenstein ring of finite Ext-index. 
Assume that every residue class field of  $R$  is algebraically closed. 
Then the polynomial ring  $R[x_1,...,x_n]$  is also of finite Ext-index.
\end{proposition}

\begin{proof} 
By Lemma \ref{finite Krull dim}, it is enough to prove that $R[x_1,...,x_n]_{\frak M}$ is of finite Ext-index for every maximal ideal $\frak M$ of $R[x_1,...,x_n]$.
Since $R$ is Artinian, we see that $\frak M \cap R=\fm$ is a maximal ideal of  $R$ and  $R/\fm$ is an algebraically closed field. 
Therefore, by Hilbert's Nullstellensatz, there are elements $r_1, \ldots , r_n \in R$  with  $\frak M = (\fm , x_1-r_1, \ldots, x_n -r_n) R[x_1, \ldots , x_n]$.  Since $R _{\fm} \cong R[x_1,...,x_n]_{\frak M} / (x_1-r_1,...,x_n-r_n)R[x_1,...,x_n]_{\frak M}$  is of finite Ext-index by Corollary \ref{cor maxmin} and since  $\{ x_1-r_1,...,x_n-r_n\}$  is a regular sequence contained in the Jacobson radical of  $R[x_1,...,x_n]_{\frak M}$,  it follows from Lemma \ref{first lemma}(4) that $R[x_1,...,x_n]_{\frak M}$  is of  finite Ext-index. 
\end{proof}

\section{Trivial Extensions}

Let $M$ be an $R$-module. 
Recall that the trivial extension  $R(M)$  of $R$  by $M$  is defined to be  $R\oplus M$  as an underlying  $R$-module that is equipped with ring structure by defining the multiplication by 
$$
(r,m)\cdot (r',m')=(rr',rm'+r'm).
$$
There are ring homomorphisms $\rho: R\longrightarrow R(M)$ with $\rho(r)=(r,0)$ and $\pi: R(M)\longrightarrow R$ with $\pi(r,m)=r$. 
Note that  $\pi\cdot\rho$ is the identity mapping on $R$. 

In this section we are mainly concerned with the trivial extension  $R(k)$  of the local ring  $(R, \fm, k)$  by the residue class field  $k$. 
We prove the following theorem as a main result of this section. 

\begin{theorem}\label{trivial ext} 
Let $(R, \fm, k)$  be a local ring and $M$, $N$ be nonzero non-free finitely generated $R(k)$-modules. 
Then  $\Tor_n^{R(k)} (M,N) \neq 0$  for all $n\geq 3$.
\end{theorem}

The following is a key to prove the theorem. 
Notice that any $R$-module can be regarded as an $R(k)$-module through  $\pi : R(k) \to R$. 

\begin{lemma}\label{key lemma} 
Let $(R, \fm, k)$  be a local ring. 
Then for $R$-modules $M$ and $N$ and for an integer $n\geq 1$  we have an isomorphism
$$
\Tor_n^{R(k)} (M,N) \cong  \Tor_n^R(M,N) \oplus \coprod_{i+j=n-1} \Tor _i^{R(k)} (M,k) \otimes_k  \Tor _j^R(k,N).
$$
\end{lemma}

\begin{proof}
Set   $A = R(k)$, $x=(0,1)\in A$ and let $\fn$ be the maximal ideal of $A$. 
Note that $\fn = (0:_Ax)$ holds and there is an isomorphism  $R \cong A/Ax$ as a ring. 
Now consider the short exact sequence of $A$-modules: 
$$
0\longrightarrow k\overset{x}\longrightarrow A\longrightarrow R\longrightarrow 0.
$$
This induces the triangle 
$$
\begin{CD}
R \otimes_A^{\mathbf{L}} k @>>> R \otimes_A^{\mathbf{L}} A 
@>{\pi}>>  R \otimes_A^{\mathbf{L}}R @>>>  R \otimes_A^{\mathbf{L}}k[1], 
\end{CD}
$$
which is actually a triangle in the derived category ${\D}^{-}(R,A)$ of right bounded chain complexes of $(R, A)$-bimodules. 
Consider the natural augmentation  $\epsilon: R \otimes_A^{\mathbf{L}} R \longrightarrow \HH _0(R\otimes_A^{\mathbf{L}}R)$ and we have the following commutative diagram in  ${\D}^{-}(R,A)$:
$$
\begin{CD}
R \otimes_A^{\mathbf{L}} A @>{\pi}>> R \otimes_A^{\mathbf{L}}R \\
 @V{\cong}VV    @V{\epsilon}VV \\  
R @>{\cong}>> \HH _0(R\otimes_A^{\mathbf{L}}R).
\end{CD}
$$
Thus $\pi$ has a left inverse $\epsilon$. 
Therefore the triangle splits off and it gives an isomorphism in  $\D^{-}(R,A)$: 
$$
(*) \qquad R \otimes_A^{\mathbf{L}}R  \cong  \ R\  \oplus \ (R\otimes_A^{\mathbf{L}}k)[1]. 
$$

Note that the natural ring homomorphism  $\rho : R \to A$  induces the forgetting functor  $\D^{-}(R,A) \to \D^{-} (R,R)$, and through this functor we can regard $(*)$ as an isomorphism in  $\D ^{-}(R,R)$.  
Now apply the functors $M\otimes_R^{\mathbf{L}}-$ from the left and $-\otimes_R^{\mathbf{L}}N$ from the right to the isomorphism $(*)$, and as a consequence we have an isomorphism in ${\D}^{-}(R,R)$: 
$$
M \otimes_A^{\mathbf{L}}N \cong  (M  \otimes_R^{\mathbf{L}} N) \ \oplus \ \left( (M\otimes_A^{\mathbf{L}}k) \otimes_k^{\mathbf{L}}(k\otimes_R^{\mathbf{L}}N)\right)[1]. 
$$
The lemma follows by taking the homology modules of the both ends. 
\end{proof}

\begin{remark} 
Let $S$ be a local ring with residue class field $\ell$ and let $M$, $N$ be $S$-modules such that $\ell_S( \Tor _n^S(M, N))< \infty$ for all $n$. 
Then we can consider the generating function  $P^S_{M,N}(t)$  defined by the equality  
$$
P^S_{M,N}(t) = \sum_{n\geq 0} \ell_S( \Tor _n^S (M,N))t^n.
$$
Recall that the Poincar\'{e} series $P^S_M(t)$  of $M$ is defined to be $P^S_{\ell ,M}(t)$ and the Poincar\'{e} series $P^S_{\ell}(t)$ is denoted  simply by $P_S(t)$.

Note that by the previous lemma we can show the equality 
$$
P^{R(k)}_{M,N}(t) = P^R_{M,N} (t) + P^{R(k)}_M (t)\  P^R_N (t) \ t, 
$$
if  $M$ and $N$ are finitely generated $R$-modules with 
 $\ell_{R(k)}( \Tor _n^{R(k)}(M, N))< \infty$ for all $n$. 
Applying this to $M=N=k$, we have 
$$
P_{R(k)}(t)=P_R(t)(1 - P_R(t)\ t)^{-1}, 
$$
which is a special case of a theorem of Gulliksen \cite[Theorem 2]{G}. 
\end{remark}

Now we proceed to the proof of the theorem.

\vspace{12pt}
\textit{Proof of Theorem \ref{trivial ext}.} 
We use the same notation as in the proof of Lemma \ref{key lemma}. 
Suppose $\Tor _n^A(M,N)=0$  for some $n\geq 3$. 
Replacing $M$ and $N$ with their first syzygies, we may assume that $\Tor _n^A(M,N)=0$  for some $n\geq 1$ and that $xM=0$ and  $xN=0$,  since  $M \subseteq \fn F$ and  $N \subseteq \fn G$ for some free $A$-modules $F$ and $G$. 
Thus we may assume $M$ and $N$ are modules over $R$ through the identification   $R \cong A /Ax$. 
Then by Lemma \ref{key lemma},  the equality $\Tor _{n-1}^A (M,k) \otimes_k (N\otimes_Rk)=0$ holds. 
Since  $N \otimes_R k \neq 0$, we see $\Tor _{n-1}^A (M,k)=0$. 
This implies that $M$ has finite projective dimension as an $A$-module. 
But $\depth (A)=0$ and by Auslander-Buchsbaum formula, $M$ is a free $A$-module. This is a contradiction. 
$\Box$

\vspace{12pt}
As applications of Theorem \ref{trivial ext} we can show the following corollaries.

\begin{corollary}\label{first cor} 
Let  $(R, \fm, k)$  be an Artinian local ring and let  $M$ and  $N$ be finitely generated modules over the trivial extension $R(k)$. 
If  $\Ext_{R(k)}^i (M, N) = 0$  for some integer $i \geq 3$, then either $M$ is $R(k)$-free or  $N$  is $R(k)$-injective. 

In particular, the equality  $\Extindex (R(k)) =0$ holds. 
\end{corollary}

\begin{proof}
Taking the Matlis dual which we denote by $(\ \ )^{\vee}$, we have $\Tor^{R(k)}_i(M, N^\vee) = 0$ for an integer $i \geq 3$. 
Hence by Theorem \ref{trivial ext},  one of $M$ and $N^{\vee}$  is $R(k)$-free. 
\end{proof}

\begin{corollary} 
Let  $(R, \fm, k)$ be an Artinian local ring. 
And let $E=E_{R(k)}(k)$ be the injective envelope of the $R(k)$-module $k$. 
Then  $\Ext^i_{R(k)} (E, R(k)) \neq 0$ for all $i \geq 3$.
\end{corollary}

\begin{proof}
Otherwise, it follows from Corollary \ref{first cor} that $E$ is $R(k)$-free or 
 $R(k)$  is $R(k)$-injective. 
In either case  $R(k)$  must be a Gorenstein ring. 
However, since the socle dimension of $R(k)$ is bigger than that of $R$ by $1$,  there is no chance for  $R(k)$  to be Gorenstein. 
\end{proof}

\begin{corollary}[Auslander-Reiten conjecture for $R(k)$]\label{AR} 
Let $(R, \fm, k)$ be an Artinian local ring and let $M$ be a finitely generated  module over  $R(k)$. 
Suppose that  $\Ext^i_{R(k)}(M, M \oplus R(k)) = 0$  for all  $i >0$  (or more weakly, for some integer $i \geq 3$). 
Then $M$ is a free $R(k)$-module.
\end{corollary}

\begin{proof}
By Corollary \ref{first cor},  either  $M$ is $R(k)$-free or $M \oplus R(k)$ is  $R(k)$-injective. 
In the latter case $R(k)$ has to be a Gorenstein ring. 
However this never occurs as we have already remarked in the proof of the previous corollary. 
\end{proof}

\section{More ring extensions}

Let $R$  be an algebra over a field $k$. 
And let $M$ be a module over the polynomial ring $R[x]$. 
The {\it specialization of $M$ to an element  $\alpha \in k$} is defined by
$$
M_{\alpha} := M \otimes_{k[x]}({k[x]}/{(x-\alpha)k[x]}).
$$
Remark that if $M$ is a finitely generated $R[x]$-module, then $M_{\alpha}$ is a finitely generated $R$-module.

\begin{lemma}\label{specialization lemma} 
Let  $R$ be a $k$-algebra and let  $\alpha\in k$. 
Assume that $x-\alpha$ is a nonzero divisor on $R[x]$-modules $M$ and $N$. 
Then we have an exact sequence
$$
0 \to \Ext^i_{R[x]}(M,N)_{\alpha}  \to \Ext^i_R(M_{\alpha},N_{\alpha}) \to \Tor_1^{k[x]}(\Ext^{i+1}_{R[x]}(M,N),\ {k[x]}/{(x-\alpha)}) \to 0 \\
$$
for each $i\geq 0$.
\end{lemma}

\begin{proof}
From the obvious exact sequence 
$$
\begin{CD}
0 @>>> N @>{x-\alpha}>> N @>>>  N_{\alpha} @>>> 0 \\
\end{CD}
$$
we have a long exact sequence
$$
\Ext_{R[x]}^i(M,N) \overset{f_i}\rightarrow \Ext_{R[x]}^i(M,N) \to \Ext_{R[x]}^i(M,N_{\alpha}) \to \Ext_{R[x]}^{i+1}(M,N) \overset{f_{i+1}}\rightarrow \Ext_{R[x]}^{i+1}(M,N), 
$$
where each $f_i$  is a multiplication mapping by  $x-\alpha$. 
Since  $x-\alpha$ is a nonzero divisor on $M$, one can easily see that 
$\Ext_{R[x]}^i (M, N_{\alpha}) \cong \Ext _R^i ( M_{\alpha}, N_{\alpha})$. 
Thus it results the exact sequence
$$
0\longrightarrow \coker (f_i) \longrightarrow \Ext _R^i (M_{\alpha},N_{\alpha})\longrightarrow \ker (f_{i+1}) \longrightarrow 0.
$$
By the definition of specialization it is easy to verify that 
$\coker (f_i) =  \Ext_{R[x]}^i (M,N)_{\alpha}$, and 
$\ker (f_{i+1}) \cong \Tor _1^{k[x]} (\Ext^{i+1}_{R[x]}(M,N), k[x]/(x-\alpha))$
\end{proof}

Related to Conjecture (E) in Section 1, we are now able to the following theorem.

\begin{theorem}\label{k(x)} 
Suppose that $k$ is an uncountable field and $R$ is a finite dimensional $k$-algebra of finite Ext-index. 
Let $k(x)$  be a transcendental extension of $k$. 
Then $R \otimes_k k(x)$ is also of finite Ext-index.
More precisely,  the inequality $\Extindex (R \otimes _{k} k(x)) \leq \Extindex (R)$  holds.
\end{theorem}

\begin{proof}
Set $b = \Extindex (R)$ and let $M'$ and $N'$ be finitely generated 
$R\otimes_k k(x)$-modules satisfying  $\Ext_{R\otimes_k k(x)}^i(M',N')=0$ for $i\gg 0$. 
We only have to show that $\Ext_{R\otimes_k k(x)}^i(M',N')=0$ for $i > b$. 

Note that  $R \otimes _k k(x)$ is just a localization of  $R[x]$  by a multiplicatively closed subset  $k[x]\backslash \{0\}$. 
Hence we can choose a finitely generated $R[x]$-submodule  $M$  of  $M'$  (resp. $N$ of $N'$) so that  $M \otimes_{k[x]}k(x) \cong M'$ (resp. $N \otimes_{k[x]}  k(x) \cong N'$). 
Notice that  $x - \alpha$ acts on  $M$ and $N$ as a non-zero divisor for each $\alpha \in k$. 
 
Since we have an isomorphism 
$\Ext^i_{R\otimes_kk(x)}(M',N') \cong \Ext^i_{R[x]}(M,N) \otimes_{k[x]}k(x)$, 
we see that $\Ext ^i_{R[x]}(M,N)\otimes_{k[x]}k(x)=0$ for $i\gg 0$.

On the other hand, since $R$ is a finite dimensional $k$-algebra, 
each module  $\Ext _{R[x]} ^i (M, N) \ (i \geq 0)$  is a finitely generated $k[x]$-module. 
Hence it has a decomposition as a $k[x]$-module as follows: 
$$
\Ext^i_{R[x]}(M,N) \cong \bigoplus_{j=1}^{s_i} k[x]/ (f_{ij}(x)) \oplus  k[x]^{r_i}, 
$$ 
where $f_{ij}(x) \neq 0 \in k[x]$. 

Since $\Ext^i_{R[x]}(M,N) \otimes_{k[x]} k(x)$ are vanishing for $i\gg 0$, we have $r_i=0$ for $i\gg 0$.
Since there are only countably many equations  $f_{ij}(x)$, we can find   an element  $\alpha\in k$ with the property $f_{ij}(\alpha)\neq 0$ for all $i, j$. 
Then,  since $x-\alpha$ acts bijectively on $k[x]/(f_{ij}(x))$, 
we see that  $\Tor_1^{k[x]}( \Ext ^{i+1}_{R[x]} (M,N), \ k[x]/(x-\alpha))=0$ for all $i$.
And we see as well that  $\Ext^i_{R[x]} (M,N)_{\alpha}=0$ for $i\gg 0$. 
Therefore the previous lemma implies that  $\Ext^i_R(M_{\alpha},N_{\alpha})=0$ for $i\gg 0$. 
Thus, by the definition of Ext-index,  we have $\Ext ^i_R(M_{\alpha}, N_{\alpha})=0$ for all $i> b$. 
Since $\Ext^i_{R[x]}(M,N)_{\alpha}$ is a submodule of $\Ext^i_R(M_{\alpha},N_{\alpha})$ by Lemma \ref{specialization lemma}, we have Ext$^i_{R[x]}(M,N)_{\alpha}=0$ for all $i> b$. 
This implies that $r_i=0$ for $i> b$, which is equivalent to the vanishing 
$
\Ext^i_{R[x]}(M,N)\otimes _{k[x]} k(x) = 0
$
for $i >b$. 
\end{proof}

\begin{remark}\label{rem for Extgap} 
Given an integer  $t \geq 1$, suppose there is a natural number $n$  with  $\Ext^i_R(M,N)=0$ for $n+1\leq i\leq n+t$ and Ext$^j_R(M,N)\neq 0$ for $j=n, n+t+1$. 
In such a case we say that $\Ext_R(M,N)$ has a gap of length $t$. 
Set
$$
\begin{aligned}
\Extgap (R):=  \sup \{ t \in \mathbb{N} \ | \ &\text{there are finitely generated  $R$-modules  $M$ and $N$} \\
&\text{ such that  $\Ext_R(M,N)$ has a gap of length $t$} \}. \\
\end{aligned}
$$
The ring $R$ is called Ext-bounded if $\Extgap (R) < \infty$. 
We should remark from \cite[Theorem 3.4(3)]{HJ} that if  $R$  is a Gorenstein local ring that is Ext-bounded, then $R$  is of finite Ext-index. 
\end{remark}

Keeping in mind this remark, we can prove the following statement completely in a similar way to the proof of Theorem \ref{k(x)}:

\begin{theorem}\label{Extgap}
Let  $R$  be a finite dimensional $k$-algebra where $k$ is an infinite field, and let  $k(x)$  be a transcendental extension of  $k$. 
If  $R$ is Ext-bounded, then so is  $R\otimes_k k(x)$. 
More precisely the inequality  $\Extgap (R \otimes _k k(x)) \leq \Extgap (R)$ holds. 
\end{theorem}

\begin{proof}
Let  $M'$  and  $N'$  be finitely generated $R \otimes _k k(x)$-modules and suppose that $\Ext _{R \otimes _k k(x)} (M', N')$  has a gap of length $t$. 
To prove  $t \leq \Extgap (R)$, we use the same notation as in the proof of Theorem \ref{k(x)}. 
First we can find finitely generated $R[x]$-modules $M$ and $N$ and an integer  $n$  satisfying $\Ext ^i_ {R[x]} (M,N) \otimes_{k[x]} k(x)=0$ for  $n+1 \leq i \leq n +t$ and  $\Ext ^j_ {R[x]} (M,N) \otimes_{k[x]} k(x)\not= 0$ for $j =  n, \ n +t+1$.
As in the proof of Theorem \ref{k(x)}, 
we decompose   $\Ext ^i_ {R[x]} (M,N)$  as $k[x]$-modules into direct sums of indecomposable ones, 
 and we have a finite number of equations  $f_{ij}(x) \ (n+1 \leq i \leq n+t+1, \  1 \leq j \leq s_i)$.
Now we choose an element  $\alpha\in k$  that is not a zero of any of these polynomials. 
Then, as in the same way as the proof of Theorem \ref{k(x)},  we can show  $\Ext ^i_R (M_{\alpha}, N_{\alpha})=0$  for $n+1 \leq i \leq n+t$. 
Thus by definition we have  $t \leq \Extgap{R}$. 
\end{proof}

\begin{remark}\label{rem for zeta}  
Let $R$ be a Cohen-Macaulay local ring with dualizing module. 
As pointed out in \cite[Observation (4.1)]{CH}, by using  maximal Cohen-Macaulay approximations,  it is easy to see that $R$ is of finite Ext-index if and only if there exists an integer $b\geq 0$, depending only on $R$, such that $P^R(M,N)\leq b$ for all maximal Cohen-Macaulay $R$-modules $M$ and $N$  with  $P^R(M,N)<\infty$. 
In fact, setting 
$$
\begin{aligned}
\zeta(R)= \sup \left\{\right. p^R(M,N) \ | \ &\text{$p^R(M,N)<\infty$ where  $M$ and $N$ are } \\
 &\text{maximal Cohen-Macaulay $R$-modules} \left.\right\}, 
\end{aligned}
$$
we can show 
$$
\zeta(R) \leq \Extindex (R) \leq \zeta (R) + d.
$$
Therefore  $R$  is of finite Ext-index if and only if  $\zeta (R) < \infty$.

Note that in the case that  $R$ ia a Gorenstein local ring, 
we have the equality  $\Extindex (R) = \zeta(R)$. 
\end{remark}

\begin{theorem}\label{zeta} 
Let  $(R, \fm, k)$ be a Cohen-Macaulay local ring of finite Ext-index. 
Assume that $R$  possesses a dualizing module and that $R$  has an uncountable coefficient field $k$. 
Then $R[x]_{\fm R[x]}$ is of finite Ext-index as well. 
More precisely, the inequality $\zeta (R[x]_{\fm R[x]}) \leq \Extindex (R)$ holds.
\end{theorem}

\begin{proof}
It is known and is easily seen that  $R[x]_{\fm R[x]}$  is also a Cohen-Macaulay  local ring with dualizing module. 
We only have to show that $\zeta (R[x]_{\fm R[x]}) \leq \Extindex (R)$. 

Let $t = \Extindex (R)$, and assume that there are maximal Cohen-Macaulay $R[x]_{\fm R[x]}$-modules  $M'$ and $N'$  such that  $\Ext ^i_{R[x]_{\fm R[x]}} (M',N')=0$ for $i\gg 0$. 
We shall show that this vanishing holds for all $i >t$. 

As in the proof of Theorem \ref{k(x)}, 
we can find finitely generated $R[x]$-modules $M$ and $N$, such that $\Ext ^i_{R[x]_{\fm R[x]}} (M',N') \cong \Ext ^i_{R[x]} (M,N)_{\fm R[x]}$ for all $i$. 
For simplicity we denote the  $R[x]$-module  $\Ext ^i_{R[x]} (M,N)$  by  $E^i$.  By Nakayama's lemma, it is easy to see that 
that  $E^i_{\fm R[x]}=0$  is equivalent to that $(E^i/\fm E^i ) \otimes_{k[x]} k(x)=0$. 

Note that each  ${E^i}/{\fm E^i}$ is a finitely generated $k[x]$-module. 
Hence, as in the proof of Theorem \ref{k(x)},  we have the decomposition as $k[x]$-modules into indecomposable modules:

$$
{E^i}/{\fm E^i} \cong \bigoplus_{j=1}^{s_i} {k[x]}/(f_{ij}(x)) \oplus k[x]^{r_i}, 
$$
where  $f_{ij} \neq 0$ are irreducible polynomials of $k[x]$.
As before we can choose an element  $\alpha\in k$ so that $f_{ij} (\alpha) \neq 0$ for all $i, j$. 
Since  $x-\alpha$ acts bijectively on $k[x]/(f_{ij}(x))$, 
 we see that  $(E^i/{\fm E^i})_{\alpha} = k ^{r_i}$  for $i$. 
By the assumption, since  $r_i=0$  holds for $i\gg 0$, 
we have  $({E^i}/{\fm E^i})_{\alpha}=0$ for $i\gg 0$. 

Note that  $({E^i}/{\fm E^i})_{\alpha} \cong {E^i_{\alpha}}/{\fm E^i_{\alpha}}$  holds by a trivial reason. 
Thus it follows from Nakayama's lemma that  $E^i_{\alpha}=0$ for $i\gg 0$. 
This means that $E^i\overset{x-\alpha}\longrightarrow E^i$ is bijective for $i \gg 0$. 
Therefore we have $\Tor _1^{k[x]}\left( E^{i+1}, k[x]/(x-\alpha) \right) = 0$ for $i\gg 0$. 

Hence  by Lemma \ref{specialization lemma}, $\Ext ^i_R ( M_{\alpha}, N_{\alpha})=0$  for $i\gg 0$, and thus  $\Ext^i_R (M_{\alpha}, N_{\alpha})=0$ for $i>t$, by the definition of  $t$. 
Use Lemma \ref{specialization lemma}, and we have that $E^i_{\alpha}=0$ for all $i>t$. 
In particular, $({E^i}/{\fm E^i})_{\alpha}=0$ for $i>t$. 
This means that  $r_i=0$ for  $i>t$ and hence  ${E^i}/{\fm E^i} \otimes_{k[x]} k(x)=0$ for $i>t$, equivalently  $E^i_{\fm R[x]} = 0$ for $i > t$. 
This shows $\zeta(R[x]_{\frak mR[x]}) \leq t$ as desired. 
\end{proof}

\begin{remark} 
As in Remark \ref{rem for Extgap}, we can prove the following statement completely in the same way as the proof above: 

\vspace{6pt}
{\it 
Let  $(R, \fm, k)$ be an Ext-bounded Cohen-Macaulay local ring admitting dualizing module. 
Assume that $R$  has an infinite coefficient field $k$. 
Then $R[x]_{\fm R[x]}$ is also Ext-bounded. 
}
\end{remark}

\bigskip 

\begin{center} \textsc{Acknowledgments}
\end{center}

The whole work of this paper was done during the visit of the first author to Department of Mathematics of Okayama University. 
The first author is grateful to Okayama University for its hospitality and facilities.



\begin{thebibliography}{99}

\bibitem{AB} 
L.~L.~Avramov and R.~-O.~Buchweitz, {\it Support varieties and cohomology over complete intersections}, Invent. Math. {\bf 142} (2000), 285--318.



\bibitem{AY} T.~Araya and Y.~Yoshino, 
{\it Remarks on a depth formula, a grade inequality and a conjecture of Auslander}, Comm. Algebra {\bf 26} (1998), no. 11, 3793--3806. 


\bibitem{BH} 
W.~Bruns and J.~Herzog,  {\it Cohen-Macaulay Rings}, Cambridge Studies in Advanced Mathematics, {\bf 39}. Cambridge University Press, Cambridge, 1993. xii+403 pp. 


\bibitem{CH} 
L.~W.~Christensen and H.~Holm, {\it Vanishing of cohomology over AC rings: conjecture of Auslander and Reiten}, preprint (2007).



\bibitem{E} 
D.~Eisenbud, {\it Commutative algebra with a view toward algebraic geometry}, 
Graduate Texts in Mathematics, {\bf 150}. Springer-Verlag, New York, 1995. xvi+785 pp. 


\bibitem{G} 
T.~H.~Gulliksen, {\it Massey operations and the Poincar\'{e} series of certain local rings}, J. Alg. {\bf 22} (1972), 223--232.



\bibitem{HJ} 
C.~Huneke and D.~A.~Jorgensen, {\it Symmetry in the vanishing of Ext over Gorenstein Rings}, Math. Scand. {\bf 93} (2003), 161--184.



\bibitem{HSV} 
C.~Huneke, L.~M.~\c{S}ega and A.~N.~Vraciu, 
{\it Vanishing of Ext and Tor over some Cohen-Macaulay local rings}, 
Illinois J. Math. {\bf 48} (2004), 295--317.

 
\bibitem{JS} 
D.~A.~Jorgensen and L. M. \c{S}ega, 
{\it Nonvanishing cohomology and classes of Gorenstein rings}, 
Adv. Math. {\bf 188} (2004), 470--490.






\bibitem{M} 
H.~Matsumura, {\it Commutative Ring Theory}, Cambridge Studies in Advanced Mathematics, {\bf 8}. Cambridge University Press, Cambridge, 1989. xiv+320 pp.



\bibitem{Mo}
I.~Mori, {\it Symmetry in the vanishing of Ext over stably symmetric algebras},   J. Algebra {\bf 310} (2007), 708--729.



\bibitem{S} 
L.~M.~$\c{S}$ega, {\it Vanishing of cohomology over Gorenstein rings of small codimension}, Proc. Amer. Math. Soc. {\bf 131} (2003), no. 8, 2313--2323 .



\bibitem{W} 
C.~A.~Weibel, {\it An introduction to homological algebra}, 
Cambridge Studies in Advanced Mathematics, {\bf 38}. 
Cambridge University Press, Cambridge, 1994. xiv+450 pp. 



\end{thebibliography}
\end{document}